\documentclass[11pt,a4paper]{article}

\usepackage{mathptmx} 
\DeclareMathAlphabet{\mathcal}{OMS}{cmsy}{m}{n}
\usepackage[top=25mm,bottom=35mm,left=20mm,right=20mm]{geometry} 

\usepackage{graphicx}
\usepackage{amssymb,amsmath,amsfonts,amsthm}
\usepackage{bm}
\usepackage[english]{babel}
\usepackage{doi}
\usepackage{mathtools}
\mathtoolsset{showonlyrefs=true}
\DeclarePairedDelimiter\paren{\lparen}{\rparen}

\newtheorem{prop}{Proposition}

\newtheorem{remark}{Remark}

\newtheorem{wf}{Weak form}
\newtheorem{scheme}{Scheme}

\newcommand{\bbR}{\mathbb{R}}
\newcommand{\px}{\partial_x}
\newcommand{\rmd}{\mathrm{d}}

\newcommand{\calH}{\mathcal{H}}

\newcommand{\fracdel}[2]{\frac{\delta #1}{\delta #2}}

\begin{document}

\title{Energy-preserving $H^1$-Galerkin schemes for the Hunter--Saxton equation}

\author{
Yuto Miyatake\thanks{Graduate School of Engineering, Nagoya University, 
\href{mailto:miyatake@na.nuap.nagoya-u.ac.jp}{miyatake@na.nuap.nagoya-u.ac.jp}}, \
Geonsik Eom, \
Tomohiro Sogabe\thanks{Graduate School of Engineering, Nagoya University,
\href{mailto:sogabe@na.nuap.nagoya-u.ac.jp}{sogabe@na.nuap.nagoya-u.ac.jp}} \
and 
Shao-Liang Zhang\thanks{Graduate School of Engineering, Nagoya University,
\href{mailto:zhang@na.nuap.nagoya-u.ac.jp}{zhang@na.nuap.nagoya-u.ac.jp}}
}


\maketitle

\begin{abstract}
We consider the numerical integration of the Hunter--Saxton equation,
which models the propagation of weakly nonlinear orientation waves.
For the equation,
we present two weak forms and their Galerkin discretizations.
The Galerkin schemes preserve the Hamiltonian of the equation and can be implemented with
cheap $H^1$ elements.
Numerical experiments confirm the effectiveness of the schemes.
\end{abstract}


\section{Introduction}
\label{sec1}

In this paper, we are concerned with the numerical integration of the Hunter--Saxton (HS) equation~\cite{hs91}:
\begin{align} \label{eq:hs}
u_{xxt} + 2u_x u_{xx} + uu_{xxx} = 0, \qquad x\in \bbR, \quad t>0,
\end{align}
where the subscript $t$ (or $x$, respectively) denotes the differentiation with respect to time variable
$t$ (or space variable $x$). 
The HS equation arises as a model for the propagation of weakly nonlinear orientation waves
in a nematic liquid crystal~\cite{hs91}.
This equation has been attracting much attention during the last two decades,
due to its rich mathematical structures and properties.
For instance, it was shown in~\cite{hz94} that the HS is integrable, has a bi-Hamiltonian structure, and possesses a Lax pair.
In particular, the first two Hamiltonians, associated to the bi-Hamiltonian structure, are
\begin{align}
\calH_1 = \int_\bbR \frac12 u_x^2 \, \rmd x, \quad \calH_2 = \int_\bbR \frac12 uu_x^2 \, \rmd x.
\end{align}
Furthermore, the HS does not have global smooth solutions, but admits conservative and dissipative global weak solutions (conservative case~\cite{hz95}, dissipative case~\cite{bc05}),
and
it can be seen as the high frequency limit of the Camassa--Holm equation~\cite{dp98}.
For more details and some generalized equations, see, for example,~\cite{da11,lenells08,w09}
and references therein.

In spite of such physical and mathematical importance,
not much effort has been devoted to the numerical integration of the HS.
This is quite a contrast to other integrable partial differential equations such as
the KdV and Camassa--Holm equations.
We are aware of only a few numerical studies:
the Enquist--Osher scheme~\cite{hs91}, an upwind finite difference scheme~\cite{hkr07},
some discontinuous Galerkin schemes~\cite{xs09,xs10},
several structure-preserving finite difference schemes~\cite{mc16}.
In general, 
the application of some standard discretization techniques to the HS seems a tough task.
The main difficulties lie on the choice of boundary conditions and 
the treatment of the term $u_{xxt}$.
\begin{itemize}
\item For a numerical simulation, we need to set appropriate space domain and boundary conditions.
We usually consider integrable PDEs on the torus, in other words, we adopt periodic boundary conditions.
However, since it is proved in~\cite{y04} that all strong solutions of the periodic HS blow-up in finite time,
we should consider other scenarios except for the case we are interested in blow-up 
phenomena.
In fact, the existing numerical studies mentioned above did not treat the periodic boundary conditions.

\item In contrast to standard expressions of evolution equations like $u_t = f(u,u_x,u_{xx},\dots )$,
there is 
the operator $\px^2$ in front of $u_t$.
Integrating both sides of the HS \eqref{eq:hs} formally, we then obtain
\begin{align}
(u_t+uu_x)_x - \frac{1}{2} u_x^2 &= a(t), \label{eq:hs1} \\
u_t + uu_x - \px^{-1} \paren*{\frac12 u_x^2 + a(t)} &= h(t) , \label{eq:hs0}
\end{align}
where $\px^{-1}$ is an inverse operator of $\px$, whose definition should depend on the problem setting,
and $a(t)$ and $h(t)$ are the integral constants, which are independent of $x$.
This discussion indicates that the HS of the form \eqref{eq:hs} has two degrees of freedom because of the two constants,
and if we try to discretize the HS of the form \eqref{eq:hs} (or \eqref{eq:hs1}), 
care must be taken for how to determine the solution.
If we focus on the form \eqref{eq:hs0}, we need to introduce a discrete counterpart of $\px^{-1}$
for the numerical integration,
which is also a nontrivial task.
\end{itemize}

In this paper, we are interested in the HS of the form
\begin{align}
u_t + uu_x - \px^{-1} \paren*{\frac12 u_x^2} = 0 , \quad x\in (-L,L), \tag{HS0} \label{hs0}
\end{align}
under the boundary condition $u(-L,t) = u_x(L,t)=0$,
where the operator $\px^{-1}$ is defined by $\px^{-1} (\cdot ) = \int_{-L}^x (\cdot ) \, \rmd x$.
This problem setting corresponds to \eqref{eq:hs0} with $a(t) = h(t)=0$,
and essentially the same setting has been considered in the literature~\cite{hkr07,mc16,xs09,xs10}.
Equivalent problem settings can also be formulated based on the expressions \eqref{eq:hs1} or \eqref{eq:hs}:
\begin{align}
(u_t+uu_x)_x - \frac{1}{2} u_x^2 = 0,  \quad x\in (-L,L), \tag{HS1}  \label{hs1} 
\end{align}
under the boundary condition $u(-L,t) = u_x(L,t)=u_x(-L,t)=0$ or $u(-L,t) = u_x(L,t)=u_{xx}(L,t)=0$;
\begin{align}
u_{xxt} + 2u_x u_{xx} + uu_{xxx} = 0,  \quad x\in (-L,L), \tag{HS2}  \label{hs2} 
\end{align}
under the boundary condition $u(-L,t) = u_x(L,t)=u_x(-L,t)=u_{xx}(L,t)=0$.
The additional boundary conditions $u_x(-L,t)=u_{xx}(L,t)=0$ are introduced
to make the solution unique.
The equivalence of these settings was proved in~\cite{mc16}.
In these settings the first Hamiltonian $\calH_1$
is constant along the solution, but the second $\calH_2$ is not (see Remark~\ref{rem1}).
Hereafter, $\calH_1$ is often referred to as the energy.

For the above problem settings, the aim of this paper is to
derive $\calH_1$-preserving Galerkin schemes for the HS,
as extensions of the $\calH_1$-preserving finite difference scheme~\cite{mc16}.
While the energy-preserving finite difference scheme is stable and produce qualitatively nice numerical solutions,
it requires the use of the uniform mesh, which is apparently less practical for some characteristic solutions of the HS (see Fig.~\ref{fig-fd} for example).
Moreover, oscillations are observed for the numerical profile $u_x$.
These drawbacks motivate us to consider Galerkin extensions of the energy-preserving finite difference scheme. 
We also note that the HS admits non-smooth weak solutions, and thus 
it seems appropriate to derive Galerkin schemes with the $H^1$ (the first order Sobolev space) framework.
 
In the Galerkin context, much effort has been devoted to establish a unified way of constructing
energy-preserving $H^1$-Galerkin schemes for evolution equations with energy-preservation property~\cite{ma08,mi12,mi14}. 
However, the methods still have their limitations and it is not straightforward to apply the methods to the HS
due to the operator $\px^2$ in front of $u_t$.
The most challenging part is to find an appropriate $H^1$-weak form with explicit energy-preservation property.
For example, Hunter and Saxton~\cite{hs91} considered a weak solution $u\in H_{\mathrm{loc}}^1 (\bbR \times \bbR^+)$:
\begin{align*}
\int_{\bbR^+}\int _\bbR \paren*{\phi_{xt} u + \frac12 \phi _{xx} u^2 - \frac12 \phi u_x^2}\,\rmd x \rmd t =0
\end{align*}
for all test functions.
But we do not adopt this expression, since it seems that the energy-preservation property  
cannot be directly established, and the second derivative exists for test functions $\phi$.
Xu and Shu~\cite{xs09} considered a weak form based on the following first order system
\begin{align}
\begin{array}{rl}
q_t + p_x - \frac12(r^2)_x =& \hspace{-0.6em} 0,\\
u_x =& \hspace{-0.6em} r, \\
r_x =& \hspace{-0.6em} q,\\
p =& \hspace{-0.6em} (ru)_x.
\end{array} \label{xswf}
\end{align}
We can derive an energy-preserving $H^1$-Galerkin scheme based on this system,
but three intermediate variables $p,q,r$ require a extra computational effort as will be 
explained in Remark~\ref{rem:cc}.
With these considerations, we present two weak formulations to the HS, and derive
energy-preserving fully-discrete Galerkin schemes.

We would also like to remark the significance of comparing two weak forms and corresponding Galerkin schemes.
In the context of energy-preserving Galerkin methods, 
finding an appropriate weak form has dominated the most important part.
However, a recent study~\cite{mi14} proposed a systematic way of finding intended weak forms
for several evolution equations
(though the technique in~\cite{mi14} cannot be directly applicable to the HS),
and the study indicates that weak forms of interest are not unique.
Hence, it is hoped that a systematic way of finding a \emph{best} weak form
is established, from the practical viewpoint.
This paper tests two different Galerkin schemes in numerical experiments
to compare the underlying weak forms, and we hope our study shall give a new insight
in the study of energy-preserving Galerkin methods.

The paper is organized as follows.
In Section~\ref{sec2}, two variational structures with respect to the first Hamiltonian $\calH_1$
and the $\calH_1$-preserving finite difference scheme~\cite{mc16} are briefly reviewed.
In Section~\ref{sec3}, two $\calH_1$-preserving $H^1$-weak forms and corresponding
Galerkin schemes are presented.
These schemes and the finite difference scheme are compared numerically in Section~\ref{sec4}.
Finally, concluding remarks are given in Section~\ref{sec5}.

\section{The variational structure and energy-preserving finite difference scheme}
\label{sec2}
In this section, we review two variational structures with respect to
the first Hamiltonian $\calH_1$.

\subsection{The variational structure and the energy-preserving finite difference scheme 
based on {\rm (\ref{hs2})}}

The variational structure \cite{hz94} and the  energy-preserving finite difference scheme~\cite{mc16},
based on the setting \eqref{hs2},
are briefly reviewed.
The variational structure to the HS, which was first introduced in \cite{hz94}, is written as
\begin{align}
u_{xxt} = (u_{xx}\px + \px u_{xx}) \px^{-2} \fracdel{\calH_1}{u}, \quad
\calH_1 (u,u_x) = \frac12 \int_{-L}^L u_x^2 \,\rmd x.
\end{align}
As in the usual interpretation, $(u_{xx}\px + \px u_{xx})$ operates on a function $f$
such that $(u_{xx}\px + \px u_{xx})f = u_{xx}f_x + (u_{xx}f)_x$.
Since $\delta \calH_1 / \delta u = -u_{xx}$, the variational formulation can be simplified to
\begin{align} \tag{HS2{$^\prime$}}\label{vari1}
u_{xxt} = - (u_{xx}\px + \px u_{xx}) u.
\end{align}
Based on this expression,
the $\calH_1$-preservation can be easily established under the boundary conditions $u(-L,t) = u_x(L,t)=u_x(-L,t) = u_{xx}(L,t)=0$:
\begin{align}
\frac{\rmd}{\rmd t}\calH_1
&= \int_{-L}^L u_x u_{xt} \, \rmd x 
= -  \int_{-L}^L uu_{xxt}\, \rmd x + [uu_{xt}]_{-L}^L
=\int_{-L}^L u \cdot (u_{xx}\px + \px u_{xx}) u \, \rmd x \\
&=
\int_{-L}^L \paren*{u u_{xx} u_x + u (u_{xx}u)_x}\,\rmd x
= 
\int_{-L}^L \paren*{u u_{xx} u_x - u_x u_{xx}u}\,\rmd x + [u u_{xx} u]_{-L}^L = 0,
\end{align}
where we used the chain rule for the first equality,
the substitution of \eqref{vari1} for the third equality, and the integration-by-parts twice.

Based on the expression \eqref{vari1}, the $\calH_1$-preserving finite difference scheme was derived in~\cite{mc16}:
\begin{align} \label{fds1}
\delta _x^{\langle 2\rangle} \frac{u_k^{(n+1)} - u_k^{(n)}}{\Delta t}
= - \paren*{\delta _x^{\langle 2\rangle}  u_k^{(n+\frac12)}} 
\paren*{\delta _x^{\langle 1\rangle}  u_k^{(n+\frac12)}}
- \delta _x^{\langle 1\rangle}  \paren*{u_k^{(n+\frac12)} \paren*{\delta _x^{\langle 2\rangle}  u_k^{(n+\frac12)}}},
\end{align}
with the discrete boundary conditions
$u_0^{(n)}=0$, $u_{-1}^{(n)} = u_1^{(n)}$, $u_{N+1}^{(n)} = u_{N-1}^{(n)}$
and $u_{N+2}^{(n)} = 2u_N^{(n)}-u_{N-2}^{(n)}$.\footnote{
We here use the following notation.
Space resp. time indices are denoted by $k$ resp. $n$.
Numerical solutions are denoted by $u_k^{(n)} \approx u(-L+k\Delta x, n \Delta t)$
on a uniform rectangular grid.
We frequently use the abbreviation: $u_k^{(n+\frac12)} = (u_k^{(n)} + u_k^{(n+1)})/2$.
The spatial domain $[-L,L]$ is partitioned so that $x_0 = -L$ and $x_N =L$.}
For the numerical solution of \eqref{fds1}, the discrete Hamiltonian
\begin{align}
\calH_{1,\rmd} (\bm{u}^{(n)}) := \sum_{k=0}^{N} {}^{\prime\prime} \frac{\Delta x}{2}
\frac{(\delta _x^+ u_k^{(n)})^2 + (\delta _x^- u_k^{(n)})^2}{2}
\end{align}
is constant independently of the time step $n$.\footnote{
The symbol
$\sum_{k=0}^{N} {}^{\prime\prime}$ stands for the trapezoidal rule:
$\sum_{k=0}^{N} {}^{\prime\prime} f_k \Delta x = \Delta x\left(
\frac12f_0 + \sum_{k=1}^{N-1} f_k +\frac12f_N \right)$.}

\begin{remark}
\label{rem1}
Under the boundary conditions $u(-L,t) = u_x(L,t)=u_x(-L,t) = u_{xx}(L,t)=0$,
the second Hamiltonian $\calH_2$ is not an invariant:
\begin{align*}
\frac{\rmd}{\rmd t} \calH_2 = \left. \frac{1}{2} u_t^2 \right| _{x=L}.
\end{align*}
This term does not vanish, unless the Dirichlet type condition is also imposed at the right boundary.
Note that adding the Dirichlet type condition makes sense only for the trivial solution $u(x,t)=0$.
\end{remark}

\subsection{The variational structure based on {\rm (\ref{hs1})}}
\label{subsec2.2}

A similar variational structure with respect to $\calH_1$ is also formulated based on the setting \eqref{hs1}:
\begin{align}
u_{xt} = \frac12(u\px + \px u) \px^{-1} \fracdel{\calH_1}{u}, \quad
\calH_1 (u,u_x) = \frac12 \int_{-L}^L u_x^2 \,\rmd x.
\end{align}
This formulation can also be simplified to 
\begin{align} \tag{HS1{$^\prime$}}\label{vari2}
u_{xt} = - \frac12 (u\px + \px u) u_x. 
\end{align}
Based on this expression,
the $\calH_1$-preservation can be established under the boundary conditions $u(-L,t) = u_x(L,t)(=u_x(-L,t) = u_{xx}(L,t))=0$:
\begin{align}
\frac{\rmd}{\rmd t}\calH_1
&= \int_{-L}^L u_x u_{xt} \, \rmd x 
=
-\frac12 \int_{-L}^L u_x u u_{xx}\,\rmd x - \frac12 \int_{-L}^L u_x (uu_x)_x\,\rmd x \\
&=
-\frac12 \int_{-L}^L u_x u u_{xx}\,\rmd x
+\frac12 \int_{-L}^L u_{xx} u u_x\,\rmd x - \frac12 \left[ uu_x^2 \right] _{-L}^L
=0.
\end{align}

In the next section, we present two weak forms
and corresponding $\calH_1$-preserving Galerkin schemes
 based on two variational structures \eqref{vari1} and \eqref{vari2}.

\section{Energy-preserving $H^1$-weak forms and their Galerkin discretizations}
\label{sec3}

In this section, two energy-preserving $H^1$-weak forms and their Galerkin discretizations
are presented.
For this aim, based on \eqref{vari1} or \eqref{vari2},
we present two $H^1$-weak forms that explicitly express
the desired $\calH_1$-preservation property,
and then discretize the weak forms so that the fully discrete schemes still keep the $\calH_1$-preservation property.

\subsection{$\calH_1$-preserving $H^1$-weak form and its Galerkin discretization
based on {\rm (\ref{vari1})}}

Let us define two function spaces $X_1$ and $X_2$ by
\begin{align}
X_1 := \{ v \in H^1 (-L,L) \ | \ v(-L) = 0\} , \quad X_2 := \{ v \in H^1 (-L,L) \ | \ v(L) = 0\} . 
\end{align}

Since the variational structure \eqref{vari1} contains the third order derivative $u_{xxx}$ in the right hand side,
we need to introduce at least one intermediate variable to consider an $H^1$ formulation.
By setting $q = u_{xx}$, we consider the following system
\begin{align*}
u_{xxt} &= -(q \px + \px q) u , \\
q &= u_{xx}.
\end{align*}
Based on this system, we consider the following weak form.\footnote{Hereafter,
we utilize the $L^2$ inner product $(f,g) = \int_{-L}^L fg\,\rmd x$.}

\begin{wf} \label{weakform}
Assume that $u(0,\cdot) $ is given in $X_1$. 
We find $u\in X_1$ and $q\in X_2$ such that, for any $v_1 \in X_1$ and $v_2 \in X_2$,
\begin{align}
\paren*{u_{xt},(v_1)_x} &= \paren*{qu_x + (qu)_x, v_1},  \label{weak1}\\
\paren*{q,v_2} &= - \paren*{u_x, (v_2)_x}. \label{weak2}
\end{align}
\end{wf}

For this weak form, the desired $\calH_1$-preservation is established.

\begin{prop} \label{prop:c}
Assume that $u_t\in X_1$.
The solution of Weak form~\ref{weakform} satisfies
\begin{align}
\frac{\rmd}{\rmd t} \calH_1 (u) = 0.
\end{align}
\end{prop}

\begin{proof}
We see
\begin{align}
\frac{\rmd}{\rmd t} \calH_1 (u) = \paren*{u_x,u_{xt}}
=\paren*{qu_x + (qu)_x, u}
= \paren*{qu_x,u} - \paren*{qu,u_x} + [qu^2]_{-L}^L = 0.
\end{align}
The first equality is just the chain rule.
The second follows from \eqref{weak1} with $v_1 = u\in X_1$.
The third is nothing but the integration-by-parts, and the last follows due to
$u(-L,t) = q(L,t)=0$.
\end{proof}

Next, we discretize Weak form~\ref{weakform} so that the resulting Galerkin scheme keeps the 
energy-preservation property.
Denoting appropriate finite dimensional subspaces of $X_1$ resp. $X_2$ by $X_{1,\rmd}$ resp. $X_{2,\rmd}$,
we consider the following Galerkin scheme.\footnote{
For Galerkin schemes,
numerical solutions are denoted by $u^{(n)} (\cdot) \approx u(\cdot, n \Delta t)$.
As is the case with the finite difference scheme, the abbreviation $u^{(n+\frac12)}  = (u^{(n)} + u^{(n+1)})/2$
is used.
}

\begin{scheme} \label{scheme1}
Suppose that $u^{(0)}$ is given in $X_{1,\rmd}$.
We find $u^{(n+1)} \in X_{1,\rmd}$ and $q^{(n+\frac12)}\in X_{2,\rmd}$ 
such that, for any $v_1 \in X_{1,\rmd}$ and $v_2 \in X_{2,\rmd}$,
\begin{align}
\paren*{\frac{u_x^{(n+1)} - u_x^{(n)}}{\Delta t},(v_1)_x} &= 
\paren*{q^{(n+\frac12)} u_x^{(n+\frac12)} + \paren*{q^{(n+\frac12)} u^{(n+\frac12)}}_x, v_1},  \label{scheme1-1}\\
\paren*{q^{(n+\frac12)},v_2} &= - \paren*{u_x^{(n+\frac12)}, (v_2)_x}. \label{scheme2-12}
\end{align}
\end{scheme}

\begin{prop}
The solution of Scheme~\ref{scheme1} satisfies
\begin{align}
 \int_{-L}^L \frac12 \paren*{u_x^{(n+1)}}^2 \,\rmd x=  \int_{-L}^L \frac12 \paren*{u_x^{(n)}}^2 \,\rmd x.
\end{align}
\end{prop}

\begin{proof}
The proof is similar to that of Proposition~\ref{prop:c}. We see
\begin{align*}
&\frac{1}{\Delta t} \paren*{ \int_{-L}^L \frac12 \paren*{u_x^{(n+1)}}^2 \,\rmd x - \int_{-L}^L \frac12 \paren*{u_x^{(n)}}^2 \,\rmd x}\\
&=
\paren*{\frac{u_x^{(n+1)} - u_x^{(n)}}{\Delta t},\paren*{\underbrace{\frac{u^{(n+1)} + u^{(n)}}{2}}_{=u^{(n+\frac12)}}}_x} \\
&=
\paren*{q^{(n+\frac12)} u_x^{(n+\frac12)} + \paren*{q^{(n+\frac12)} u^{(n+\frac12)}}_x, u^{(n+\frac12)}} \\
&=
\paren*{q^{(n+\frac12)} u_x^{(n+\frac12)}, u^{(n+\frac12)}} 
-
\paren*{q^{(n+\frac12)} u^{(n+\frac12)}, u_x^{(n+\frac12)}} 
+ \left[ q^{(n+\frac12)} \paren*{u^{(n+\frac12)}}^2 \right]_{-L}^L =0.
\end{align*}
Note that, though the chain rule is replaced by just the factorization for the first equality,
the subsequent procedures are almost the same as the proof of Proposition~\ref{prop:c}.
\end{proof}

\subsection{$\calH_1$-preserving $H^1$-weak form and its Galerkin discretization
based on {\rm (\ref{vari2})}}

Let us seek another possibility of constructing an $\calH_1$-preserving $H^1$-Galerkin scheme
based on the variational structure \eqref{vari2}.
In this case, though \eqref{vari2} contains only up to second order derivatives,
it seems necessary to introduce an intermediate variable to establish the $\calH_1$-preservation.
By setting $r=u_x$, we divide the HS into a system
\begin{align*}
u_{xt} &= -\frac12 (u\px + \px u)r ,\\
r&= u_x,
\end{align*}
and consider the following weak form.

\begin{wf} \label{weakform2}
Assume that $u(0,\cdot) $ is given in $X_1$. 
We find $u\in X_1$ and $r\in X_2$ such that, for any $v_1 \in X_2$ and $v_2 \in X_1$,
\begin{align}
\paren*{u_{xt},v_1} &= -\frac12 \paren*{ur_x + (ur)_x, v_1},  \label{weak2-1}\\
\paren*{r,(v_2)_x} &=  \paren*{u_x, (v_2)_x}. \label{weak2-2}
\end{align}
\end{wf}

Note that we here require the test functions $v_1$ and $v_2$ are in $X_2$ and $X_1$, respectively,
in order to ensure the $\calH_1$-preservation.

\begin{prop} \label{prop:c2}
Assume that $u_t\in X_1$.
The solution of Weak form~\ref{weakform2} satisfies
\begin{align}
\frac{\rmd}{\rmd t} \calH_1 (u) = 0.
\end{align}
\end{prop}

\begin{proof}
\begin{align}
\frac{\rmd}{\rmd t} \calH_1 (u) = \paren*{u_x,u_{xt}}
=\paren*{u_{xt},r}
= -\frac12 \paren*{ur_x,r} - \frac12\paren*{(ur)_x,r}
= -\frac12 \paren*{ur_x,r} + \frac12 \paren*{ur,r_x} - \frac12 \left[ ur^2 \right]_{-L}^L 
= 0.
\end{align}
The second follows from \eqref{weak2-2} with $v_2 = u_t\in X_1$.
The third is due to \eqref{weak2-1} with $v_1=r\in X_2$,
the forth is obtained by the integration-by-parts, and the last follows due to
$u(-L,t) = r(L,t)=0$.
\end{proof}

We now present an $\calH_1$-preserving Galerkin scheme based on Weak form~\ref{weakform2}.

\begin{scheme} \label{scheme2}
Suppose that $u^{(0)}$ is given in $X_{1,\rmd}$.
We find $u^{(n+1)} \in X_{1,\rmd}$ and $r^{(n+\frac12)}\in X_{2,\rmd}$ 
such that, for any $v_1 \in X_{2,\rmd}$ and $v_2 \in X_{1,\rmd}$,
\begin{align}
\paren*{\frac{u_x^{(n+1)} - u_x^{(n)}}{\Delta t},(v_1)_x} &= 
 -\frac12 \paren*{u^{(n+\frac12)} r_x^{(n+\frac12)} + \paren*{u^{(n+\frac12)} r^{(n+\frac12)}}_x, v_1},  \label{scheme2-1}\\
\paren*{r^{(n+\frac12)},(v_2)_x} &=  \paren*{u_x^{(n+\frac12)}, (v_2)_x}. \label{scheme2-2}
\end{align}
\end{scheme}

\begin{prop}
The solution of Scheme~\ref{scheme2} satisfies
\begin{align}
 \int_{-L}^L \frac12 \paren*{u_x^{(n+1)}}^2 \,\rmd x=  \int_{-L}^L \frac12 \paren*{ u_x^{(n)}}^2 \,\rmd x.
\end{align}
\end{prop}

\begin{proof}
The proof is similar to that of Proposition~\ref{prop:c2}.
\begin{align*}
&\frac{1}{\Delta t} \paren*{ \int_{-L}^L \frac12 \paren*{u_x^{(n+1)}}^2 \,\rmd x - \int_{-L}^L \frac12 \paren*{u_x^{(n)}}^2 \,\rmd x}\\
&=
\paren*{\frac{u_x^{(n+1)} - u_x^{(n)}}{\Delta t},\paren*{\underbrace{\frac{u^{(n+1)} + u^{(n)}}{2}}_{=u^{(n+\frac12)}}}_x} \\
&=
\paren*{\frac{u_x^{(n+1)} - u_x^{(n)}}{\Delta t},r^{(n+\frac12)}} \\
&=
-\frac12
\paren*{u^{(n+\frac12)} r_x^{(n+\frac12)} + \paren*{u^{(n+\frac12)} r^{(n+\frac12)}}_x, r^{(n+\frac12)}} \\
&=
-\frac12
\paren*{u^{(n+\frac12)} r_x^{(n+\frac12)}, r^{(n+\frac12)}} 
+ \frac12
\paren*{u^{(n+\frac12)} r^{(n+\frac12)}, r_x^{(n+\frac12)}} 
-\frac12 \left[ u^{(n+\frac12)} \paren*{r^{(n+\frac12)}}^2 \right]_{-L}^L =0.
\end{align*}
Note that, though the chain rule is replaced by the factorization for the first equality,
the subsequent procedures are almost the same as Proposition~\ref{prop:c2}.
\end{proof}

\subsection{Implementation}
We make a brief comment on the implementation of the schemes.

We denote the basis functions of $X_{1,\rmd}$ and $X_{2,\rmd}$ by
$\phi _k (x)$ ($k=0,\dots,M-1$) and $\varphi _k (x)$ ($k=0,\dots,N-1$).
We introduce the coefficient vectors
$\bm{u}^{(n)} = (u_0^{(n)}, u_1^{(n)},\dots, u_{M-1}^{(n)} )^\top \in \bbR^M$ and
$\bm{q}^{(n+\frac12)} = (q_0^{(n+\frac12)}, q_1^{(n+\frac12)},\dots, q_{N-1}^{(n+\frac12)} )^\top \in\bbR^N$ for
the functions $u^{(n)} =\sum_{k=0}^{M-1}u_k^{(n)} \phi_x \in X_{1,\rmd}$ and
$q^{(n)} =\sum_{k=0}^{N-1}q_k^{(n)} \varphi_x \in X_{2,\rmd}$, respectively.
By using the matrices
\begin{align} \label{comat}
A_{ij} = \paren*{(\phi_i)_x,(\phi_j)_x} , \quad 
B_{ij} = \paren*{\varphi_i,\varphi_j} , \quad
C_{ij} = \paren*{(\phi_i)_x,(\varphi_j)_x} ,
\end{align}
Scheme~\ref{scheme1} can be written as
\begin{align}
A \paren*{\frac{\bm{u}^{(n+1)} - \bm{u}^{(n)}}{\Delta t}} &= g_1\paren*{\bm{u} ^{(n+\frac12)}, \bm{q} ^{(n+\frac12)}}, \\
B \bm{q} ^{(n+\frac12)} &= -C^\top \bm{u} ^{(n+\frac12)},
\end{align}
where $g_1$ denotes the nonlinear part.
But, since the matrix $B$ is invertible, the two equations immediately reduce to
\begin{align}
A \paren*{\frac{\bm{u}^{(n+1)} - \bm{u}^{(n)}}{\Delta t}} &= g_1\paren*{\bm{u} ^{(n+\frac12)}, -B^{-1}C^\top\bm{u} ^{(n+\frac12)}}, \label{mat1}
\end{align}
which is a nonlinear system for $\bm{u}^{(n+1)}$.
Therefore, we need not compute the intermediate variable $\bm{q} ^{(n+\frac12)}$,
and the dimension of the system to be solved is reduced from $M+N$ to $M$.

\begin{remark}\label{rem:cc}
An $\calH_1$-preserving $H^1$-Galerkin scheme can also be derived based on the 
system \eqref{xswf}.
Similar to the above discussion, we see that the dimension of the system to be solved is reduced to
the dimension of the function space to which $u^{(n)}$ belong.
However, in the reduced system, there are at least three matrix inverses,
which requires extra effort for function evaluations. 
\end{remark}

For Scheme~\ref{scheme2},
we also define a matrix $D$ 
by $D_{ij} = \paren*{\varphi_i , (\phi_j)_x}$.
We assume $M=N$ and the matrix $D$ is nonsingular.
Then, Scheme~\ref{scheme2} can be written as
\begin{align}
D \paren*{\frac{\bm{u}^{(n+1)} - \bm{u}^{(n)}}{\Delta t}} &= g_2\paren*{\bm{u} ^{(n+\frac12)}, \bm{r} ^{(n+\frac12)}}, \\
D^\top \bm{r} ^{(n+\frac12)} &= A \bm{u} ^{(n+\frac12)},
\end{align}
where $g_2$ denotes the nonlinear part.
The two equations immediately reduce to
\begin{align}
D \paren*{\frac{\bm{u}^{(n+1)} - \bm{u}^{(n)}}{\Delta t}} &= g_2\paren*{\bm{u} ^{(n+\frac12)}, D^{-\top}A\bm{u} ^{(n+\frac12)}}, \label{mat2}
\end{align}
which is of size $M$, as is the case with \eqref{mat1}.
However, the computational complexity is much reduced compared to \eqref{mat1}.
For simplicity, let us assume that
$X_{1,\rmd}$ and $X_{2,\rmd}$
consists of P1 elements, and $\phi_k$ and $\varphi_k$ are standard basis functions.
In this case, while the matrix $A$ is a tridiagonal matrix whose elements depend on the grid points,
the matrix $D$ is given by
\begin{align}
D = \frac12 \begin{bmatrix}
1 &  &  &  &  &  &\\ 
0 & 1 &  &  &  &  &\\ 
-1 & 0 & \ddots &  & &  &\\ 
 & -1 & \ddots & \ddots & & &\\ 
 &  & \ddots & \ddots & \ddots & & \\ 
 &  &  & -1 & 0 & 1 & \\
 & & & & -1 & 0& 1
\end{bmatrix},
\end{align} 
and thus its inverse $D^{-1}$ can be explicitly given in advance.
Hence, though the systems \eqref{mat1} and \eqref{mat2} are both of size $M$,
the latter can be implemented easily,
in the sense that the function evaluation and the computation of the Jacobian of $g_2$ are straightforward,
and the computational cost is also reduced.

\section{Numerical experiments}
\label{sec4}

In this section, through some numerical experiments, 
we compare three schemes: the finite difference scheme~\eqref{fds1},
Scheme~\ref{scheme1} and Scheme~\ref{scheme2}.
All the computations were done in the computation environment:
3.5 GHz Intel Core i5, 8GB memory, OS X 10.10.5.
We used MATLAB (R2015a).
Nonlinear equations were solved by \emph{fsolve} with tolerance $10^{-16}$.
As a reference solution, we consider the following exact solution to the HS
\begin{align}
u(x,t) = \begin{cases}
0 & \text{if} \quad x\leq 0 \\
x/(0.5t+1) & \text{if} \quad 0<x<(0.5t+1)^2 \\
0.5t + 1 & \text{if}  \quad (0.5t+1)^2\leq x .
\end{cases}
\end{align}
This is a solution for not only the original HS \eqref{eq:hs} but also the HS on the bounded domain $x\in [-L,L]$ in the time interval $t\in [0,2(\sqrt{L}-1))$.

Figs.~\ref{fig-fd},~\ref{fig-galerkin1}~and~\ref{fig-galerkin2} show 
the numerical results obtained by the finite difference scheme~\eqref{fds1},
Scheme~\ref{scheme1} and Scheme~\ref{scheme2}, respectively.
For the Galerkin schemes, we used uniform meshes and employed the P1 elements.
For all schemes,
the numerical profiles of $u(x,t)$ seem to agree with the exact one.
However, there are some minor differences.
The numerical solution of the finite difference scheme~\eqref{fds1} is more smooth around
$x=2.25$ than Scheme~\ref{scheme1}.
For Scheme~\ref{scheme2}, there are small oscillations in the region $x\in[2.25,6]$.
These differences can be understood from the numerical profiles of $u_x(x,t)$.
For the finite difference scheme~\eqref{fds1} and Scheme~\ref{scheme2},
large oscillations are observed. 
They are not eliminated even for smaller mesh sizes.
It should be noted that the oscillation obtained by Scheme~\ref{scheme1} is bounded small enough
despite special techniques, such as upwinding flux for discontinuous Galerkin methods, are not employed.
For all schemes, the Hamiltonian is well preserved. For Scheme~\ref{scheme1},
results by explicit and implicit Euler methods with the same spatial discretization are also plotted
for the sake of comparison.
The explicit integrator is unstable, and 
the energy dissipation of the implicit integrator indicates that the qualitative behaviour is deteriorated
as time passes.

We also compare the schemes from the viewpoint of the error growth by means of the $L^\infty$ norm in Fig.~\ref{fig:uerror}.
In the left and central figures, results by all schemes with the use of uniform mesh are plotted.
In the right figure, we compare the use of uniform mesh with non-uniform mesh for Scheme~\ref{scheme1}.
It is observed that the use of non-uniform mesh significantly reduces the $L^\infty$ error. 

These numerical results indicate that Scheme~\ref{scheme1} is better than Scheme~\ref{scheme2}.

\begin{figure}
\centering
\includegraphics{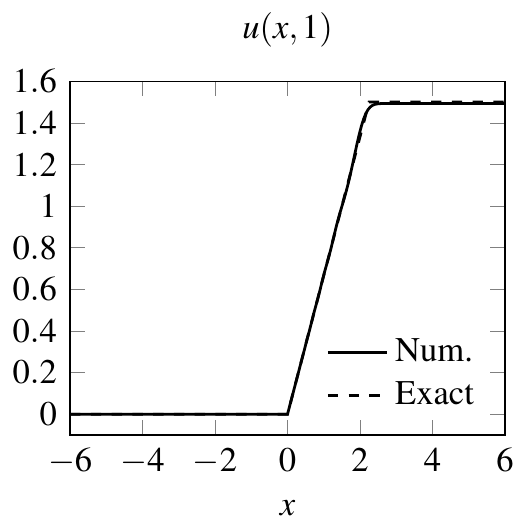}
\includegraphics{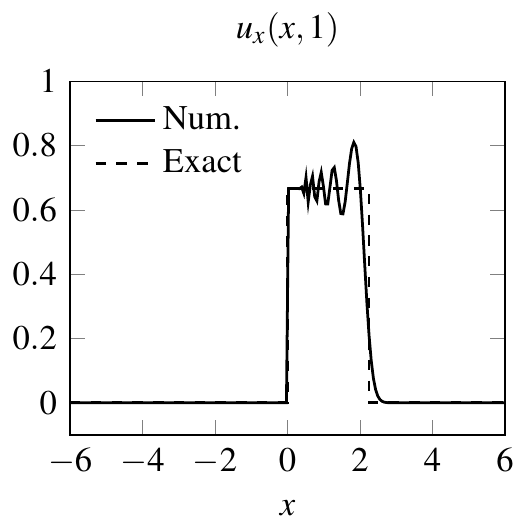}
\includegraphics{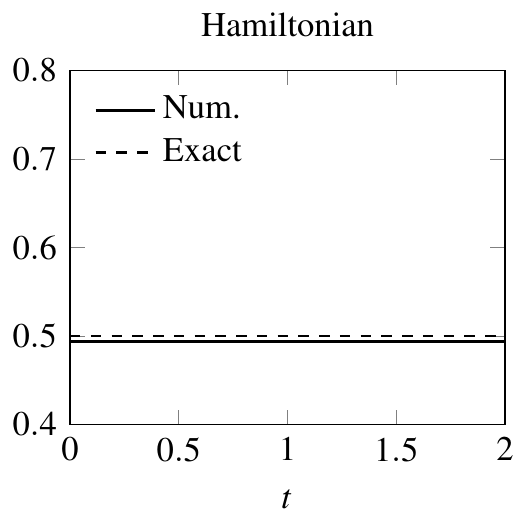}
\caption{Scheme~\protect\eqref{fds1}: exact and numerical profiles of $u(x,t)$ and $u_x(x,t)$ at time $T_{\mathrm{end}} = 1$
and computed Hamiltonian ($\Delta x = 12/200$ and $\Delta t=0.01$).}
\label{fig-fd}
\end{figure}

\begin{figure}
\centering
\includegraphics{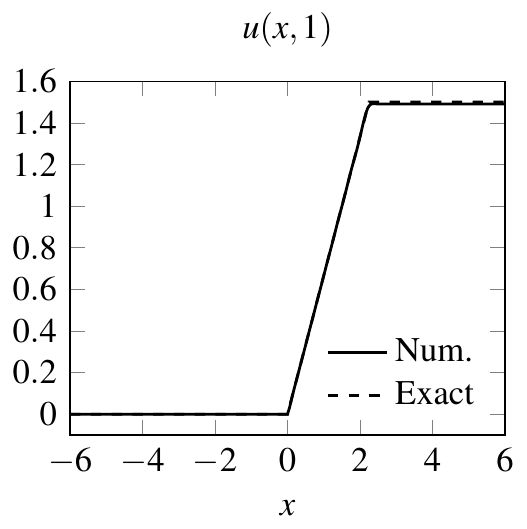}
\includegraphics{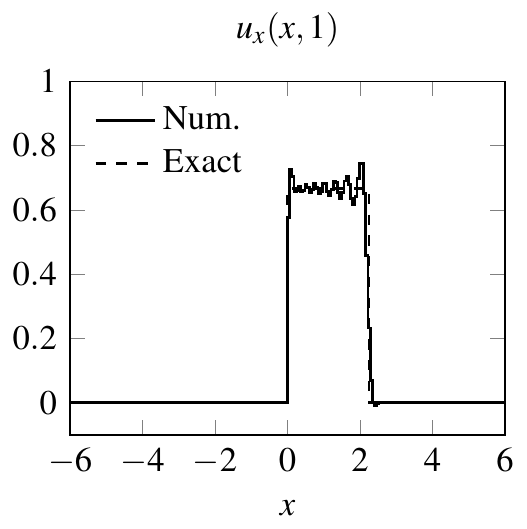}
\includegraphics{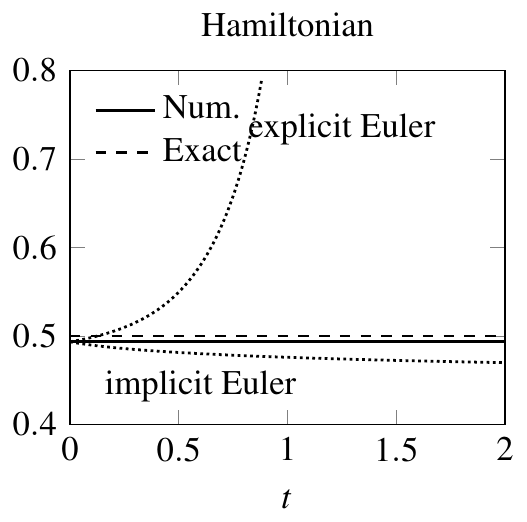}
\caption{Scheme~\protect\ref{scheme1}: exact and numerical profiles of $u(x,t)$ and $u_x(x,t)$ at time $T_{\mathrm{end}} = 1$
and computed Hamiltonian ($\Delta x = 12/200$ and $\Delta t=0.01$).
In the right figure, results by explicit and implicit Euler methods with the same spatial discretization are also presented.}
\label{fig-galerkin1}
\end{figure}

\begin{figure}
\centering
\includegraphics{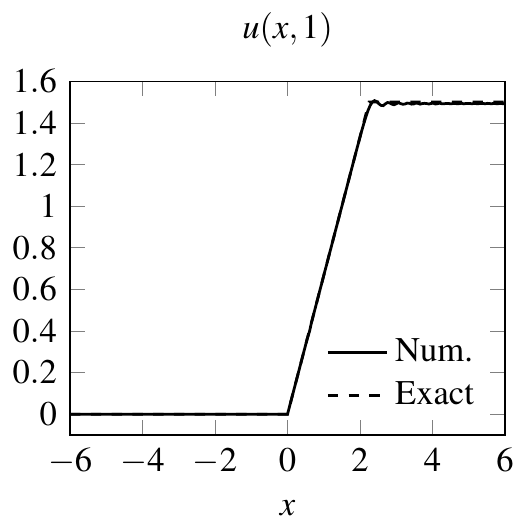}
\includegraphics{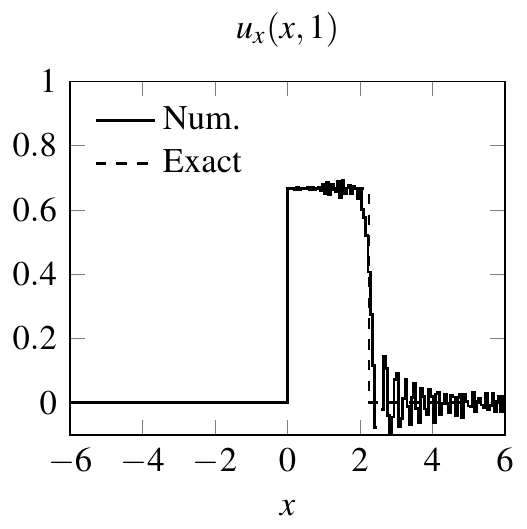}
\includegraphics{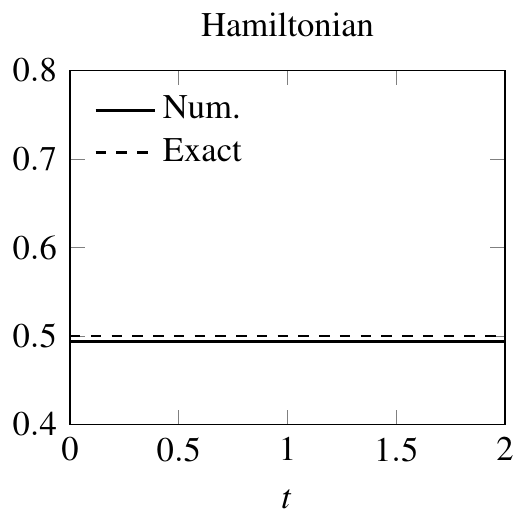}
\caption{Scheme~\protect\ref{scheme2}: exact and numerical profiles of $u(x,t)$ and $u_x(x,t)$ at time $T_{\mathrm{end}} = 1$
and computed Hamiltonian ($\Delta x = 12/200$ and $\Delta t=0.01$).}
\label{fig-galerkin2}
\end{figure}

\begin{figure}
\centering
\includegraphics{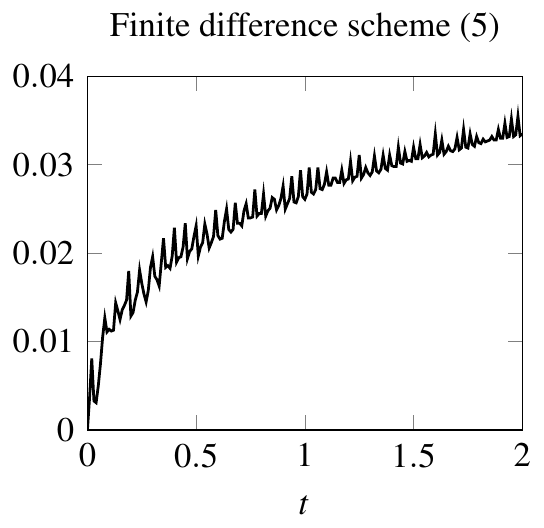}
\includegraphics{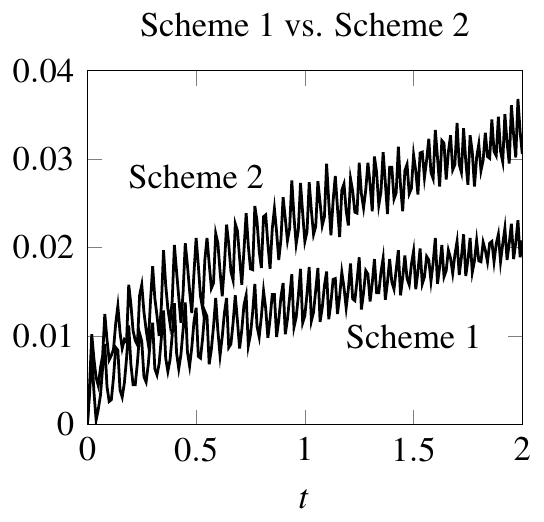}
\includegraphics{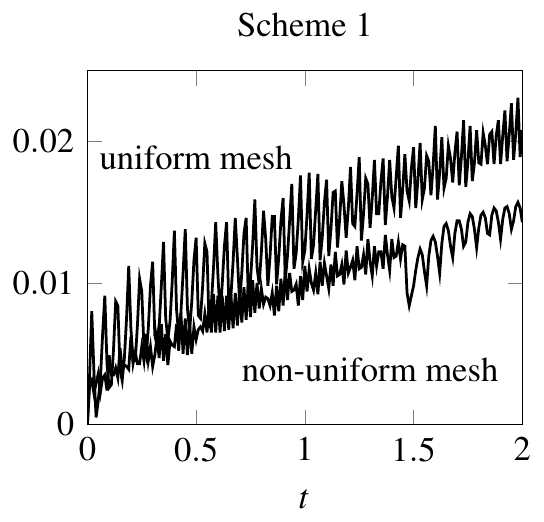}
\caption{$L^\infty$ errors ($\Delta t = 0.01$).
In the left and central figures, results by all schemes with uniform mesh ($\Delta x = 12/200$) are plotted.
In the right figure, results by Scheme~\protect\ref{scheme1} with uniform mesh and non-uniform mesh
(Fig.~\protect\ref{fig-grid}) are compared.}
\label{fig:uerror}
\end{figure}

\begin{figure}
\centering
\includegraphics{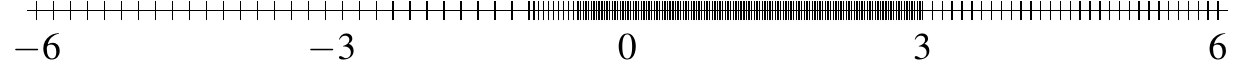}
\caption{Non-uniform mesh $N=200$.}
\label{fig-grid}
\end{figure}

\section{Concluding remarks}
\label{sec5}

We presented two $\calH_1$-preserving $H^1$-weak forms and their Galerkin discretizations for the HS equation.
Even if we use uniform meshes,
the first Galerkin scheme (Scheme~\ref{scheme1}) seems superior to the finite difference scheme~\eqref{fds1} and
Scheme~\ref{scheme2}.
Results obtained by Scheme~\ref{scheme1} can be further improved by using the non-uniform meshes.

For Scheme~\ref{scheme1}, the numerical profile of $u_x$ 
was much better than usually expected to non-smooth solutions.
Mechanism behind this phenomena should be investigated with more numerical experiments
or from more theoretical viewpoints.

As far as the authors' knowledge, in the context of energy-preserving Galerkin methods,
finding an intended weak form has been a main subject, and
there are few studies where several energy-preserving weak forms and their Galerkin discretizations
are compared.
It is hoped that the fact that substantial differences were observed for the HS
offers a new research direction in the literature.

\bibliographystyle{plain}
\bibliography{bibhs}
\end{document}